\documentclass{amsart}
\usepackage{amsfonts,amssymb,amscd,amsmath,enumerate,verbatim,calc}
\usepackage[all]{xy}

\newcommand{\TFAE}{The following conditions are equivalent:}
\newcommand{\tFAE}{the following conditions are equivalent:}
\newcommand{\CM}{Cohen-Macaulay}

\newcommand{\ff}{\text{if and only if}}
\newcommand{\wrt}{with respect to}
\newcommand{\B}{\mathcal{B} }

\newcommand{\I}{\mathbb{I} }
\newcommand{\n}{\mathfrak{n} }
\newcommand{\m}{\mathfrak{m} }
\newcommand{\M}{\mathfrak{M} }

\newcommand{\R}{\mathcal{R}(I) }
\newcommand{\Z}{\mathbb{Z} }

\newcommand{\rt}{\rightarrow}
\newcommand{\xar}{\longrightarrow}
\newcommand{\ov}{\overline}

\newcommand{\bx}{\mathbf{x}}
\newcommand{\by}{\mathbf{y}}

\newcommand{\wt}{\widetilde }
\newcommand{\si}{\sigma }

\newcommand{\fg}{\operatorname{fg}}

\newcommand{\grade}{\operatorname{grade}}
\newcommand{\depth}{\operatorname{depth}}

\newcommand{\red}{\operatorname{red}}

\newcommand{\amp}{\operatorname{amp}}

\theoremstyle{plain}

\newtheorem{theorem}{Theorem}[section]
\newtheorem{corollary}[theorem]{Corollary}

\newtheorem{proposition}[theorem]{Proposition}

\theoremstyle{definition}
\newtheorem{definition}[theorem]{Definition}
\newtheorem{s}[theorem]{}
\newtheorem{remark}[theorem]{Remark}
\newtheorem{example}[theorem]{Example}

\theoremstyle{remark}

\numberwithin{equation}{theorem}

\begin{document}

\title[Higher Associated graded modules]{Ratliff-Rush Filtration, regularity and \\ depth of Higher Associated graded modules \\ Part II}
 \author{Tony~J.~Puthenpurakal}
\date{\today}
\address{Department of Mathematics, Indian Institute of Technology Bombay, Powai, Mumbai 400 076}

\email{tputhen@math.iitb.ac.in}

\thanks{The author was partly supported by IIT Bombay seed grant 03ir053}



\begin{abstract}
 Let $(A,\m)$ be a Noetherian local ring, let $M$ be a finitely generated \CM \ $A$-module of dimension $r \geq 2$ and let $I$ be an ideal of definition
 for $M$. Set $L^I(M) = \bigoplus_{n\geq 0}M/I^{n+1}M$. In part one of this paper we showed that $L^I(M)$ is a module over $\R$,
  the Rees algebra of $I$ and we gave many applications of $L^I(M)$  to study  the associated graded module, $G_I(M)$.
In this paper we give many further applications of our technique;  most notable is a reformulation of a classical result due to Narita in terms of
the Ratliff-Rush filtration.     This reformulation can be extended to all dimensions $\geq 2$.
\end{abstract}

\maketitle

\tableofcontents

\section*{Introduction}
Dear Reader; while reading this paper it is a good idea to have  part 1 of this paper \cite{Pu5}.
Let $(A,\m)$ be a Noetherian local ring with residue field $k = A/\m$. Let $M$ be a finitely generated \CM \ $A$-module of dimension $r \geq 2$ and let $I$ be an ideal of definition
for $M$ i.e., $\lambda (M/IM)$ is finite.   Here $\lambda(-)$ denotes length.         Let $G_I(A) = \bigoplus_{n\geq 0}I^n/I^{n+1}$ be the associated graded ring of $A$ \wrt \ $I$ and let $G_I(M)= \bigoplus_{n\geq 0}I^nM/I^{n+1}M$ be the associated graded module of $M$ \wrt \ $I$.

Set $L^I(M) = \bigoplus_{n\geq 0}M/I^{n+1}M$. In part one of this
paper we showed that $L^I(M)$ is a    module over $\R$;
the Rees-algebra of $I$. It is \emph{not finitely generated} as a $\R$-module. In part 1 we gave
 applications of $L^I(M)$ in the study of associated graded modules. We have collected these properties in section \ref{Lprop}.

 \emph{Applications } \\
 In part 1 of this paper we gave \emph{five} applications of  the technique of $L^I(M)$ in the study of $G_I(M)$. In part 2 we give \emph{six} more applications of our technique.

\noindent\textbf{VI.} Let  $x$ be $M$-superficial \wrt \  $I$. Set $N = M/xM$ and $u =xt \in R(I)$.
We say the Ratliff-Rush filtration on $M$ \wrt \ $I$ \emph{behaves well mod $x$ }  if
$$\ov{\wt{I^{n} M}} = \wt{I^{n} N} \quad \text{ for all} \ n \geq 1.  $$
We prove that the  Ratliff-Rush filtration on $M$ \wrt \ $I$ behaves well mod $x$ \ff \ $H^{1}(L^{I}(M)) = 0$; see Theorem \ref{h1zero}. In particular our result proves that
 if Ratliff-Rush filtration behaves well mod one superficial element then it does so \emph{with any} superficial element.

We then relate vanishing of $H^{i}(L^{I}(M))$ for $i = 1,\ldots, s$ to good behavior of Ratliff-Rush filtration mod a superficial sequence of length $s$; see Theorem \ref{rrMODsup}. Thus good behavior of the Ratliff-Rush filtration mod a superficial sequence  is a cohomological property.

 \noindent\textbf{VII.} \emph{minimal} $\I$- \emph{invariant} :

 Recall that we say $G_I (M)$ is generalized \CM \ module if
\begin{equation*}
\lambda (H^i (G_I (M))) < \infty \;\mbox{for}\; i = 0, 1,\cdots, r
- 1.
\end{equation*}
For generalized \CM \ module the St\"{u}ckrad-Vogel invariant
\begin{equation*}
\I(G_I (M)) = \sum^{r - 1}_{i = 0} \binom{r - 1}{i}
\lambda (H^i (G_I (M)))
\end{equation*}
plays a crucial role. If $x^* \in G_I (A)_1$ is $G_I (M)$- regular then
one can verify
\begin{equation*}
\I(G_I (M/xM)) \leq  \I(G_I (M)).
\end{equation*}
So in some sense if we have to study minimal $\I$- invariant then
we have to first consider the case when  $ \depth G_I (M) = 0.$   In Theorem \ref{vogel} we
prove that if $G_I
(M)$ is generalized \CM \  and  $ \depth G_I (M) = 0$ then
\begin{equation*}
\I(G_I (M)) \geq r \bullet \lambda \left( H^0 (G_I (M)) \right).
\end{equation*}
We also prove that the following are equivalent
\begin{enumerate}[\rm (i)]
\item
$\I (G_I (M)) = r \bullet \lambda (H^0 (G_I (M)))$
\item
$H^i (L^I (M)) = 0$ for $i = 1, 2, \cdots, r - 1.$
\item
 The Ratliff-Rush
filtration on $M$ behaves well mod superficial sequences (of length $r -1$).
\end{enumerate}

\noindent\textbf{VIII.}  A classical result, due to Narita \cite{Nar}  states that if
 $(A, \m)$ is \CM \  of $\dim 2$ then
\begin{equation*}
e^I_2 (A) = 0 \ \ \mbox{iff}\ \  \red (I^n) = 1 \;\mbox{for all} \ n \gg 0.
\end{equation*}
This  can be easily extended to Cohen-Macaulay modules of
dimension two.  However  Narita's result fails (even for \CM \
rings) in dimension $\geq 3$; see \ref{CPR}.
We first reformulate Narita's result in dimension 2.

 Let $\wt{G}_I(M) = \bigoplus_{n \geq 0} \wt{I^nM}/\wt{I^{n+1}M}$ be the associated graded module of the Ratliff-Rush filtration. We prove
$$e^I_2 (M) = 0  \Longleftrightarrow \widetilde{G_I} (M) \ \text{ has minimal multiplicity.} $$
This reformulation can be generalized. We prove that if $ \dim M \geq 2$ then
 $$e^I_i (M) = 0 \ \text{for  \ }  i = 2,\ldots, r  \Longleftrightarrow \widetilde{G_I} (M) \ \text{ has minimal multiplicity.} $$

\noindent\textbf{IX.} Assume $G_I (M)$ is Generalized \CM. In Theorem \ref{asymp} we give an explicit computation of
$H^i( G_{I^n}(M))$ for $n \gg 0$. This is  in terms of
$H^i(L^I(M))_{-1}$ for $i = 1,\ldots, r-1$.

\noindent\textbf{X.}
Set $$\xi_I(M) = \lim_{n \to \infty} \depth G_{I^n}(M).$$
By a result of Elias, for the case $M = A$, this limit exists. In part 1 of the paper we proved that this limit also exists for \CM \ modules.
In this paper we prove, see Theorem \ref{xi-mod-x-large}, that if $x$ is $M$-superficial \wrt\ $I$ then
$$ \xi_I(M/x^sM) \geq \xi_I(M) -1 \quad \text{for all} \ s \gg 0. $$
We also give an example which shows that \emph{strict inequality} can occur above. We \textbf{do not } know whether
$\xi_I(M/xM) \geq \xi_I(M) -1 $.

\noindent\textbf{XI.}
 Let $M$ be a \CM \ module of dimension $3$ and let $\red_I (M) = 2.$  In Theorem \ref{red2} we prove  $e^I_3 (M) \leq
0.$ We also show that
\[
     e^I_3 (M) = 0   \Longleftrightarrow         \xi_I(M) \geq 2.
\]

\section{Notation and Preliminaries}
In this section we introduce some notation and discuss a few preliminaries which will
be used in this paper.
In this paper all rings are commutative Noetherian and all modules (\emph{unless stated otherwise})
are assumed finitely generated. We  use  terminology from  \cite{BH}.
Let
$(A,\m)$ be a local ring of dimension $d$ with residue field $k =
A/\m$. Let $M$ be \CM \  $A$-module of dimension $r$. Let $I$ be an ideal
 ideal of definition for $M$.

\s If $p \in M$ is non-zero and  $j$ is the largest integer such that $p \in I^{j}M$,
then we let $p^*$ denote the image of $p$
 in $I^{j}M/I^{j+1}M$.

\s
\label{hilbcoeff}
 The \emph{Hilbert function} of $M$ with respect to $I$ is the function
\[
 H_{I}(M,n) =  \lambda(I^nM/I^{n+1}M)\quad \text{for all} \ n \geq 0.
\]
It is well known that the formal power series $\sum_{n \geq 0}H_{I}(M,n)z^n$
represents a rational function of a
special type:
\begin{equation*}
\sum_{n \geq 0}H_{I}(M,n) z^n = \frac{h_{I}(M ,z)}{(1-z)^{r}}\quad \text{where}
\ r = \dim M \ \text{and} \ h_{I}(M,z) \in\mathbb{Z}[z].
\end{equation*}
 Set
$e_{i}^{I}(M) = (h_I(M,z))^{(i)}(1)/i! $ for all $i\geq 0$. The integers $e_{i}^{I}(M)$ are called \emph{Hilbert
coefficients} of $M$ with
respect to $I$.
The number $e_{0}^{I}(M)$ is also called the \emph{multiplicity} of $M$ with respect
to $I$.
Set $\chi_{1}^{I}(M) = e_{1}^{I}(M) - e_{0}^{I}(M) + \lambda(M/IM)$.

\s Assume $r = \dim M >0$. Since $M$ is \CM \ we give the following equivalent definition of superficial elements and superficial sequences.
 Let $x \in I$. We say $x$ is \emph{$M$-superficial \wrt \ $I$} if $(I^{n+1}M \colon_M x) = I^nM$ for all
 $n \gg 0$.

 Assume $i \leq r$. Let $\bx = x_1,\ldots, x_i \in I \setminus I^2$. We say $\bx$ is an $M$-superficial sequence \wrt \ $I$; if $x_1$ is $M$-superficial \wrt \ $I$, $x_2$ is $M/x_1M$-superficial \wrt \ $I$,$\cdots$, $x_r$ is
$M/(x_1,\ldots,x_{r-1})M$-superficial \wrt \ $I$.

\s The advantage of dealing  of working with modules is that we do not have to change rings while going mod superficial elements. This we do. However the following remark is relevant.
\begin{remark}
Let $x_1,... ,x_s$ be a sequence in  $ I$ and
set $J= (x_1,... ,x_s)$.
Set $B = A/J$,  $K = I/J$ and  $N = M/JM$. Notice
\[
G_I(N) = G_K(N) \quad \text{and} \quad \depth_{ G_I(A)} G_I(N) = \depth_{G_K(B)} G_K(N).
\]
\end{remark}

\s\textbf{ Associated graded module and Hilbert function mod a superficial element: }\label{mod-sup-h} \\
Let $x \in I$
be $M$-superficial.  Set $N = M/xM$. The following is well-known cf., \cite{Pu1}.
\begin{enumerate}[\rm(1)]
  \item Set $b_I(M,z) = \sum_{i \geq 0}\ell\big( (I^{n+1}M \colon_M x)/ I^nM \big)z^n$. Since
  $x$ is $M$-superficial we have $b_I(M,z) \in \Z[z]$.
  \item $h_I(M,z) = h_I(N,z) - (1-z)^rb_I(M,z)$; cf., \cite[Corollary 10 ]{Pu1}.
  \item So we have
  \begin{enumerate}[\rm(a)]
    \item $e_i(M) = e_i(N)$ for $i = 0,\ldots, r-1$.
    \item $e_r(M) = e_r(N) - (-1)^r b_I(M,1)$.
  \end{enumerate}
  \item
   The following are equivalent
   \begin{enumerate}[\rm(a)]
     \item $x^*$ is $G_I(M)$-regular.
     \item $G_I(N) = G_I(M)/x^*G_I(M)$
     \item $b_I(M,z) = 0$
     \item $e_r(M) = e_r(N)$.
   \end{enumerate}
   \item
   \emph{(Sally descent)} If $\depth G_I(N) > 0$ then $x^*$ is $G_I(M)$-regular.
\end{enumerate}

\s \textbf{The Ratliff-Rush filtration and its Hilbert function:}
For definition of Ratliff-Rush filtration and some basic properties see \cite[section 2]{Pu5}.
We assume $\dim M > 0$.
Since $\wt{I^{n}M} = I^{n}M$ for all $n \gg 0$ we get that the function
$\wt{H}_{I}(M,n) =  \lambda( \wt{I^{n}M}/\wt{I^{n+1}M})$ is
a polynomial function. This is the Hilbert function of
$\wt{G}_I(M) = \bigoplus_{n \geq 0} \wt{I^nM}/\wt{I^{n+1}M}$; the associated graded module of the
Ratliff-Rush filtration on $M$.
As usual set
\[
\sum_{n \geq 0}\lambda( \wt{I^{n}M}/\wt{I^{n+1}M})z^n  = \frac{\wt{h}_{I}(M,z)}{(1-z)^r}; \ \text{where $\wt{h}_{I}(M,z) \in \Z[z]$.}
\]
 Set $\wt{e}_{i}^{I}(M) = (\wt{h}_{I}(M,z))^{(i)}(1)/i!$ the \emph{Hilbert coefficients}
of the Ratliff-Rush filtration of $M$ \wrt \ $I$.

 \s \label{rrhilbertcoeff}\emph{Relation between Hilbert coefficients of Ratliff-Rush filtration and the usual $I$-adic filtration
 on $M$}\\
 For all $n \geq 0 $ we have the following exact sequence
 \begin{equation}\label{rr-usual}
 0 \xar \frac{\wt{I^{n+1}M}}{I^{n+1}M} \xar \frac{M}{I^{n+1}M} \xar \frac{M}{\wt{I^{n+1}M}} \xar 0
 \end{equation}
   $$ \text{Set} \quad  r_I(M,z) = \bigoplus_{n \geq 0} \lambda\left(\frac{\wt{I^{n+1}M}}{I^{n+1}M}\right)z^n.$$
  Notice        $r_I(M,z) \in  \Z[z]$.
   Using \ref{rr-usual} we get
 \[
 h_I(M,z) = \wt{h_I}(M,z) + (1-z)^{r+1}r_I(M,z).
 \]
 Therefore we have
 \begin{enumerate}[\rm (a)]
   \item
 $e_{i}^{I}(M) = \wt{e}_{i}^{I}(M)$ for $0 \leq i \leq r$.
 \item
 $e_{r+1}^{I}(M) = \wt{e}_{r+1}^{I}(M) + (-1)^{r+1}r_I(M,1). $
 \item
 The following are equivalent
 \begin{enumerate}[\rm (i)]
 \item
 $\depth G_I(M) > 0$.
 \item
 $r_I(M,z) = 0$.
 \item
 $ e_{r+1}(M) = \wt{e_{r+1}}(M)$.
 \end{enumerate}
 \end{enumerate}

\s \textbf{Base change:}
\label{AtoA'}
In our arguments we do use  a few base changes. See \cite[1.4]{Pu5}
for details.
\section{Some Properties of $L^{I}(M)$}\label{Lprop}

In this section we collect all the properties of $L^{I}(M)$ which we proved in \cite{Pu5}. Throughout thus section
$(A,\m)$ is a Noetherian local ring with infinite residue field, $M$ is a \emph{\CM }\ module of dimension $r \geq 1$ and $I$ is \emph{an ideal of definition} for
$M$.

\s \label{mod-struc} Set $\R = A[It]$;  the Rees Algebra of $I$. In \cite[4.2]{Pu5} we proved that \\
$L^{I}(M) = \bigoplus_{n\geq 0}M/I^{n}M$ is a $\R$-module.

 \s Set $\M = \m\oplus R(I)_+$. Let $H^{i}(-) = H^{i}_{\M}$ denote the $i^{th}$-local cohomology functor \wrt \ $\M$. Recall a graded $\R$-module $L$ is said to be
*-Artinian if
every descending chain of graded submodules of $L$ terminates. For example if $E$ is a finitely generated $\R$-module then $H^{i}(E)$ is *-Artinian for all
$i \geq 0$.

\s \label{zero-lc} In \cite[4.7]{Pu5} we proved that
\[
H^{0}(L^I(M)) = \bigoplus_{n\geq 0} \frac{\wt{I^{n+1}M}}{I^{n+1}M}.
\]
\s \label{Artin}
For $L^I(M)$ we proved that for $0 \leq i \leq  r - 1$
\begin{enumerate}[\rm (a)]
\item
$H^{i}(L^I(M))$ are  *-Artinian; see \cite[4.4]{Pu5}.
\item
$H^{i}(L^I(M))_n = 0$ for all $n \gg 0$; see \cite[1.10 ]{Pu5}.
\item
 $H^{i}(L^I(M))_n$  has finite length
for all $n \in \mathbb{Z}$; see \cite[6.4]{Pu5}.
\item
$\lambda(H^{i}(L^I(M))_n)$  coincides with a polynomial for all $n \ll 0$; see \cite[6.4]{Pu5}.
\end{enumerate}

\s \label{I-FES} The natural maps $0\rt I^nM/I^{n+1}M \rt M/I^{n+1}M \rt M/I^nM \rt 0 $ induce an exact
sequence of $R(I)$-modules
\begin{equation}
\label{dag}
0 \xar G_{I}(M) \xar L^I(M) \xrightarrow{\Pi} L^I(M)(-1) \xar 0.
\end{equation}
We call (\ref{dag}) \emph{the first fundamental exact sequence}.  We use (\ref{dag}) also to relate the local cohomology of $G_I(M)$ and $L^I(M)$.

\s \label{II-FES} Let $x$ be  $M$-superficial \wrt \ $I$ and set  $N = M/xM$ and $u =xt \in \R_1$. Notice $L^I(M)/u L^I(M) = L^I(N)$. Let $I = (x_1,\ldots,x_m)$. There exists  $c_i \in A$ such that $x
= \sum_{i= 1}^{s}c_ix_i$. Set $X = \sum_{i= s}^{l}c_iX_i$.
For each $n \geq 1$ we have the following exact sequence of $A$-modules:
\begin{align*}
0 \xar \frac{I^{n+1}M\colon x}{I^nM} \xar \frac{M}{I^nM} &\xrightarrow{\psi_n} \frac{M}{I^{n+1}M} \xar \frac{N}{I^{n+1}N} \xar 0, \\
\text{where} \quad \psi_n(m + I^nM) &= xm + I^{n+1}M.
\end{align*}
This sequence induces the following  exact sequence of $\R$-modules:
\begin{equation}
\label{dagg}
0 \xar \B^{I}(x,M) \xar L^{I}(M)(-1)\xrightarrow{\Psi_X} L^{I}(M) \xrightarrow{\rho^x}  L^{I}(N)\xar 0,
\end{equation}
where $\Psi_X$ is left multiplication by $X$ and
\[
\B^{I}(x,M) = \bigoplus_{n \geq 0}\frac{(I^{n+1}M\colon_M x)}{I^nM}.
\]
We call (\ref{dagg}) the \emph{second fundamental exact sequence. }

\s \label{long-mod} Notice  $\lambda\left(\B^{I}(x,M) \right) < \infty$. A standard trick yields the following long exact sequence connecting
the local cohomology of $L^I(M)$ and
$L^I(N)$:
\begin{equation}
\label{longH}
\begin{split}
0 \xar \B^{I}(x,M) &\xar H^{0}(L^{I}(M))(-1) \xar H^{0}(L^{I}(M)) \xar H^{0}(L^{I}(N)) \\
                  &\xar H^{1}(L^{I}(M))(-1) \xar H^{1}(L^{I}(M)) \xar H^{1}(L^{I}(N)) \\
                 & \cdots \cdots \\
               \end{split}
\end{equation}

\s \label{Artin-vanish} We will use the following well-known result regarding *-Artinian modules quite often:

Let $L$ be a *-Artinian $\R$-module.
\begin{enumerate}[\rm (a)]
\item
If $\psi \colon L(-1) \rt L$ is a monomorphism then $L = 0$.
\item
If $\phi \colon L \rt L(-1)$ is a monomorphism then $L = 0$.
\end{enumerate}

\s \label{deg-vanish}
 A criterion for the vanishing of $H^{i}(L^I(M))_n$, the $n^{th}$-graded
component of $H^{i}(L^I(M))$ is the following (see \cite[8.3]{Pu5}):
Suppose for some $n$ and some $i$ with $0 \leq i \leq r -1$ the maps
\begin{align*}
H^i(\Pi)_n &\colon H^i\left(L^I(M) \right)_n \xar H^i\left(L^I(M) \right)_{n-1} \ \ \text{and} \\
H^i(\Psi_X)_n &\colon H^i\left(L^I(M) \right)_{n-1} \xar H^i\left(L^I(M) \right)_{n}  \ \ \text{are injective.}
\end{align*}
 Then $H^i(L^I(M))_n = 0$.

 \s \label{power-of-I} Set
 \begin{align*}
\xi_I(M) &:= \underset{0 \leq i \leq r-1}\min\{ \ i \ \mid H^{i}(L)_{-1} \neq 0 \ \text{or} \ \ell(H^{i}(L)) = \infty \}.\\
\amp_I(M) &:= \max\{\ |n| \ \mid H^{i}(L)_{n-1} \neq 0 \  \text{for} \ i = 0,\ldots, \xi_I(M) - 1 \}.
\end{align*}
In \cite[7.5]{Pu5} we showed that
\[
 \depth G_{I^l}(M) = \xi_I(M) \ \text{for all} \ l > \amp_I(M).
\]

\section{ Ratliff-Rush filtration mod a superficial element}

Let $M$ be an $A$-module with
 $\grade(I,M) \geq 2$.
Let $x \in I \setminus I^2$ is an
$M$-superficial \wrt \ $I.$ Set $N = M/xM$. Let $\rho^x \colon  M \rt N$ be the natural map.
Notice $\rho^x \left(\wt{I^{n} M} \right) \subseteq \left(\wt{I^{n} N} \right) \quad \text{for all} \ n \geq 1.$
We say the Ratliff-Rush filtration on $M$ \wrt \ $I$ \emph{behaves well mod $x$ }  if
$\rho^x \left(\wt{I^{n} M} \right)= \left(\wt{I^{n} N} \right)$ for all $n \geq 1$.
In this section we show that this is equivalent to $H^1(L^I(M)) = 0$.   Thus if the Ratliff-Rush filtration behaves well \wrt \
one superficial element then it behaves    well \wrt \ any superficial element.

\s \label{august}
 Notice $\rho^x$ induces the  maps
\[
\rho^x_n \colon  \frac{\wt{I^{n} M}}{I^n M} \longrightarrow \frac{\wt{I^{n} N}}{I^n N}  \quad \quad \quad   \text{for all} \ n \geq 1.
\]
Fix $n \geq 1$. It can be easily checked that  $\rho^x \left(\wt{I^{n} M} \right) = \left(\wt{I^{n} N} \right)$ \ff \ $\rho_n^x$ is \emph{surjective}.
The map  $\rho^x$ induces a natural $\R$-linear map $L^{I}(M) \rt L^I(N)$ which we also denote by $\rho^x$. Consider the induced map
\begin{align}
H^{0}(\rho^x) \colon H^0(L^I(M)) &\longrightarrow  H^0(L^I(N))  \\
\label{basicObs}\text{Notice} \quad H^{0}(\rho^x)_n &= \rho^x_n \ \ \text{for all} \ n \in \Z; \quad \text{(use \ref{zero-lc})}.
\end{align}

\s    Let $x$ be an $M$-superficial element \wrt \ $I$ and set $N = M/xM$.
  We
 have the following
exact sequence
\begin{equation}
\label{supexN}
 0 \xar \frac{(I^{n+1}M\colon_M x)}{I^nM} \xar \frac{\wt{I^nM}}{I^nM} \xrightarrow{\alpha_{n-1}^{x}}\frac{\wt{I^{n+1}M}}{I^{n+1}M} \xrightarrow{\rho_n} \frac{\wt{I^{n+1}N}}{I^{n+1}N};
\end{equation}
Here $\rho_n$ is the natural quotient map (defined since $\ov{\wt{I^nM}} \subseteq \ \wt{I^nN} $ for all $n \geq 0$.)
\begin{equation}
\label{supexN0}
\text{In particular} \quad \rho_0 \colon\frac{\wt{IM}}{IM} \xar \frac{\wt{IN}}{IN}  \quad \text{is injective.}
\end{equation}

The following result  shows that if Ratliff-Rush filtration behaves well mod one superficial element then it does so \emph{with any} superficial element.
\begin{theorem}\label{h1zero}
Let $(A,\m)$ be local with an infinite residue field and let $M$ be a \CM \ $A$-module of dimension $r \geq 2$. Let $I$ be an ideal of definition for $M$.
Let $x$ be  $M$-superficial  \wrt \ $I$. \TFAE
\begin{enumerate}[\rm (i)]
  \item The  Ratliff-Rush filtration on $M$ \wrt \ $I$ behaves well mod $x$.
  \item $H^{1}(L^{I}(M)) = 0$.
\end{enumerate}
\end{theorem}
\begin{proof}
($ \Longrightarrow$)  By hypothesis and \ref{august} we get that  $\rho_n^x$ is \emph{surjective} for all $n \geq 1$. Using \ref{basicObs} we get that
the $\R$-linear map
 $$ H^0(\rho^x) \colon H^0 (L^I (M)) \rightarrow H^0 (L^I (M_1)) \quad \text{is surjective.}$$
  Using \ref{longH} we get that
 the $\R$-linear map
$$ H^1(\Psi^x) \colon H^1 (L^I (M)) (-1) \rightarrow H^1 (L^I (M)) \quad \text{ is injective}.$$
As $H^1 (L^I (M))$ is $*$-Artinian
we get $H^1 (L^I (M)) = 0$;  see (\ref{Artin-vanish}).

($ \Longleftarrow  $) If $H^1 (L^I (M)) = 0$ then the map $H^0(\rho^x )$ is surjective.
 Using \ref{basicObs} we get that  $\rho_n^x$ is \emph{surjective} for all $n \geq 1$.
So by \ref{august} it follows that the Ratliff-Rush filtration on $M$ \wrt \ $I$ behaves well mod $x$.
\end{proof}
As an immediate corollary we get
\begin{corollary}\label{invar-1}[with hypothesis as in \ref{h1zero}]
Suppose $x,y \in I $ are two \textbf{distinct elements} which are $M$-superficial \wrt \ $I$. \TFAE
\begin{enumerate}[\rm (i)]
  \item The  Ratliff-Rush filtration on $M$ \wrt \ $I$ behaves well mod $x$.
  \item The  Ratliff-Rush filtration on $M$ \wrt \ $I$ behaves well mod $y$.
  \end{enumerate}
\qed
\end{corollary}

  \section{Two Examples}
  In this section we discuss two examples in detail. The examples are
\begin{enumerate}
\item
 A \CM \ module $M$ of dimension $2$ with
  $e_{2}^{I}(M) = 0$.
\item
    $M = A$ has dimension 2, the ideal $I$ is integrally closed and $e_{2}^{I}(A) = e_1^I(A) - e_0^I(A) + \ell(A/I)$.
  \end{enumerate}
  We prove that in both these cases $H^1(L^I(M)) = 0$. So by \ref{h1zero} the module $M$ behaves well \wrt \ a superficial element. By \ref{power-of-I} we also
  have that $G_I(M)$ is generalized \CM. We compute the St\"{u}ckrad-Vogel  invariant $\I(G_I(M))$ for these two examples. We also prove many preliminary results which we  need. \emph{These preliminary results
 are "well known" when $M =A$. Lack of a suitable reference has compelled me to  include it here.}

The following proposition  gives a convenient representation of the
 Hilbert coefficients of the Ratliff-Rush filtration when $M$ is \CM \ with $\dim M = 1$ or $2$.    It is a generalization of the corresponding ring
  case; see
\cite[Theorem 3]{It},
 \cite[p.\ 300]{RV2},
\cite[Equation 11]{RV3} and
\cite[1.9]{GuRossi}.

\begin{proposition}
\label{sigma}
Let $(A,\m)$ be a local ring,
  $M$ a Cohen-Macaulay   $A$-module  of dimension $r = 1 $ or $2$ and let $I$ be an ideal of definition
of $M$. Let $J = (x_1,x_{\dim M})$ be a minimal reduction of $M$ \wrt \ $I$.
Set $\sigma_{i}^{I}(M) = \lambda(\wt{I^{j+1}M}/J\wt{I^{j}M})$.  We  have
\begin{enumerate}[\rm (1) \quad]
\item
$\Delta^{r-1}\left( \wt{H}_{I}(M,n) \right) = e^{I}_{0}(M) - \sigma_{n}^{I}(M)$.
\item
 $\wt{h}_{M}^{I}(z) = (e^{I}_{0}(M) -  \sigma_{0}^{I}(M)) + \sum_{i\geq 1}\left( \sigma_{i-1}^{I}(M) -  \sigma_{i}^{I}(M)\right)z^i$.
\item
$\wt{e}^{I}_{k}(M)= \sum_{j \geq k -1}\binom{j}{k-1}\sigma_{j}^{I}(M)$ for each $k \geq 1$.
\end{enumerate}
\end{proposition}
\begin{proof}
Clearly (2) follows from (1) and (3) follows from (2). When $\dim M = 1$ the proof
given in  \cite[Equation 2]{RV4} for $M =A$ can be easily generalized.

When $\dim M = 2$ we use a technique due to Huneke \cite[2.4]{Hun}. Consider the
exact sequence:
\[
0 \xar \frac{M}{\wt{I^nM} \colon (x_1,x_2)} \xrightarrow{\alpha} \left(\frac{M}{\wt{I^nM}}\right)^2 \xrightarrow{\beta} \frac{JM}{J\wt{I^nM}} \xar 0,
\]
where
\begin{align*}
\alpha( m + \wt{I^nM} \colon (x_1,x_2) &=  (-x_2m + \wt{I^nM}, x_1m + \wt{I^nM}) \ \text{and} \\
\beta(m_1 + \wt{I^nM}, m_2 + \wt{I^nM}) &= x_1m_1 + x_2m_2 + J\wt{I^nM}.
\end{align*}
Notice that
\[
\wt{I^nM} \colon (x_1,x_2) = \wt{I^{n-1}M}, \quad \text{and} \ \lambda(JM/J\wt{I^nM}) = \lambda(M/J\wt{I^nM}) - e_{0}^{I}(M).
\]
So by using the exact sequence above we get that
\begin{align*}
2\lambda(M/\wt{I^nM}) &= \lambda(M/\wt{I^{n-1}M}) +  \lambda(M/J\wt{I^nM}) - e_{0}^{I}(M) \\
                      &= \lambda(M/\wt{I^{n-1}M}) +  \lambda(M/\wt{I^nM}) + \sigma_{n}^{I}(M)  - e_{0}^{I}(M).
\end{align*}
So it follows that $ e_{0}^{I}(M)  =  \wt{H}_{I}(M,n) - \wt{H}_{I}(M,n-1) + \sigma_{n}^{I}(M)$.
This proves (1).
\end{proof}

An easy consequence of the previous proposition is the following:
\begin{corollary}
\label{sigmac}
(With the same hypothesis as above)
We have
\begin{enumerate}[\rm (1)]
\item
$\chi_{1}^{I}(M) \geq  \lambda(\wt{IM}/IM )$ with equality \ff \ $\wt{I^{j+1}M} =J\wt{I^{j}M}$ for all $j \geq 1$.
\item
When $M = A$ and $I$ is integrally closed then
$ \wt{e_2}^{I}(A) \geq \chi_{1}^{I}(A).$
\end{enumerate}
\begin{proof}
Both the assertions follow from Proposition \ref{sigma} and  the fact
$\sigma_{0}^{I}(M) = e_{0}^{I}(M) - \lambda(M/IM) + \lambda(\wt{IM}/IM ).$
\end{proof}
\end{corollary}

\begin{proposition}
\label{e2}
Let $(A,\m)$ be a Noetherian local ring with infinite residue field,  $M$  a $2$-dimensional \CM \ $A$-module and let
$I$ be an $\m$-primary ideal. Let $x,y$ be an  $M$-superficial sequence \wrt \ $I$. Set
$J =(x,y)$ and $N = M/xM$.
If $e_{2}^{I}(M) = 0$ then
\begin{enumerate}[\rm (i)]
\item
$\wt{I^{i+1}M} = J\wt{I^{i}M}$  for all   $i \geq 1$.
\item
$\wt{I^{i+1}M } \subseteq I^iM$ for all  $i \geq 1$.
\item
$\ov{\wt{I^iM}} = \wt{I^iN}$ for all $i \geq 1$.
\end{enumerate}
\end{proposition}
\begin{proof}
Set $\si_{i} = \lambda\left(\wt{I^{i+1}M}/J\wt{I^{i}M}\right)$ for $i \geq 0$.
Since $ \sum_{i \geq 1}i\si_i = e_{2}^{I}(M) $ we get that $\si_i = 0$ for all $i \geq 1$.
Therefore
\begin{equation}
\label{p1e1}
 \wt{I^{i+1}M} = J\wt{I^{i}M} \ \  \text{ for all} \  i \geq 1.
\end{equation}
This proves (i). The assertion (ii) follows easily from (i).

(iii) Using Corollary \ref{sigmac}(1) we get
\begin{equation}
\label{p1e2}
\chi_{1}^{I}(M) = \lambda\left(\wt{IM}/IM\right).
\end{equation}
Using \ref{mod-sup-h}(3), \ref{sigmac}(1) and (\ref{supexN0})  we get the following inequalities:
\begin{equation}
\label{p1e3}
 \chi_{1}^{I}(M) = \chi_{1}^{I}(N) \geq \lambda\left(\wt{IN}/IN\right) \geq \lambda\left(\wt{IM}/IM\right).
\end{equation}

Therefore by equations (\ref{p1e2}) and (\ref{p1e3}) we have

(a) $\lambda\left(\wt{IM}/IM\right) = \lambda\left(\wt{IN}/IN\right)$.

(b) $ \chi_{1}^{I}(N) = \lambda\left(\wt{IN}/IN\right)$.

(c)  From (a) and (\ref{supexN0})  we get that $\ov{\wt{IM}} = \wt{IN}$.

(d)  From (b) and \ref{sigmac}(1)  we get that   $\wt{I^{i+1}N} = y\wt{I^{i}N}$ for all $i \geq 1$.

We get the required result  from  (\ref{p1e1}), (c), and (d).
\end{proof}

\begin{proposition}
\label{chi2}
Let $(A,\m)$ be a $2$-dimensional \CM \  local ring with infinite residue field and let
$I$ be an $\m$-primary integrally closed ideal. Let $x$ be $A$-superficial \wrt \ $I$ such that
$ I/(x)$ is integrally closed ideal in $B= A/(x)$.
If $e_{2}^{I}(A) =  e_{1}^{I}(A) - e_{0}^{I}(A) + \lambda(A/I)$ then
\begin{enumerate}[\rm (i)]
\item
$\wt{I^{i+1}} = J\wt{I^{i}}$    for all   $i \geq 2$.
\item
$\wt{I^{i+1}} \subseteq I^i$ for all $i \geq 1$.
\item
 $\ov{\wt{I^iA}} = \wt{I^iB}$ for all
$i \geq 1$.
\end{enumerate}
\end{proposition}
\begin{proof}
The first assertion follows
from the expression of $e_{1}^{I}(A)$ and $e_{2}^{I}(A)$
as given in  \ref{sigma}.

(ii) We prove  this by induction.
First note that $\wt{I^2} \subseteq \wt{I}$. Since $I$ is integrally closed
we have $\wt{I} = I$. So the result holds for $i =1$.
 We assume it for $i = l - 1$ and prove for $i= l$ (here
$l \geq 2$).
For $l \geq 2$ by
(i) we have  $\wt{I^{l+1}} = J\wt{I^{l}}$. Since by induction hypothesis $ \wt{I^l}\subseteq I^{l-1}$
we get that
 $\wt{I^{l+1}} \subseteq I^l$. This proves (ii).

 (iii) We use part of
 (\ref{longH})
\begin{equation*}
0 \xar \B^{I}(x,A) \xar H^0(L^{I}(A))(-1)\xar H^0(L^{I}(A))  \xar H^0(L^{I}(B)).
\end{equation*}
Therefore we have that
\begin{equation}
\label{p2e1}
b =  \lambda\left(\B^{I}(x,A)\right) \leq  \lambda\left(\  H^0(L^{I}(B)) \right) = r.
\end{equation}
Note that if $b = r $ the natural map $H^0(L^{I}(A))  \xar H^0(L^{I}(B) $ is surjective and this will
imply $\ov{\wt{I^iA}} = \wt{I^iB}$ for all
$i \geq 1$.

Using \ref{mod-sup-h}(3)  we have
\begin{align*}
 e_{2}^{I}(B) &= e_{2}^{I}(A) + b \\
              &=   e_{1}^{I}(A) - e_{0}^{I}(A) + \lambda(A/I) + b \\
              &=   e_{1}^{I}(B) - e_{0}^{I}(B) + \lambda(B/IB) + b.
\end{align*}

Note that $ e_{2}^{I}(B) = \wt{ e_{2}}^{I}(B) +  r $.
Therefore
\[
\wt{ e_{2}}^{I}(B)  =  e_{1}^{I}(B) - e_{0}^{I}(B)   + \lambda(B/IB) + b - r.
\]
Since $I/(x)$ is integrally closed ideal  of $B$ we get from  Corollary \ref{sigmac}(2)
that
$ \wt{ e_{2}}^{I}(B) \geq  e_{1}^{I}(B) - e_{0}^{I}(B) + \lambda(B/IB) $.
Therefore   $b \geq r$. Using this and (\ref{p2e1}) we conclude that $b = r$ and so the
 result follows.
\end{proof}

\begin{theorem}
\label{e2dim2main}
Let $A$ be local and let
$M$ be  a finite Cohen-Macaulay $A$-module of dimension $ r = 2$. Let $I$ be an
ideal of definition for $M$.
 Assume \emph{ any one of the following} conditions hold
\begin{enumerate}[\rm (1)]
\item
$e_{2}^{I}(M) = 0$
\item
$M = A$, the ideal $I$ is integrally closed and $e_{2}^{I}(A) = e_{1}^{I}(A)-  e_{0}^{I}(A) + \ell(A/I)$.
\end{enumerate}
Then
\begin{enumerate}[\rm (a)]
\item
$H^1(L^I(M)) = 0$.
\item
 $G_{I^n}(M)$ is \CM \ for all $n \gg 0$.
 \item
  $G_I(M)$ is generalized \CM \ with $\mathbb{I}(G_I(M)) = -2e_{3}^{I}(M)$.
  \item
  If $I = \m$ then
$G_{\m}(M)$ is a quasi-Buchsbaum $G_{\m}(A)$-module.
\end{enumerate}
\end{theorem}
\begin{proof}
By using  \cite[1.4]{Pu5}, we may pass to ring $A' = A[X]_{\m A[X]}$. This we do.
So we may assume that the residue field of $A$ is infinite and that if $K$ is an $\m$-primary
integrally closed ideal in $A$ then there exists an element $x \in I$ which is $A$-superficial
\wrt \  $I$ such that, the $A/(x)$ ideal $I/(x)$ is integrally closed.

\noindent\textbf{ The case when } $e_{2}^{I}(M) = 0 :$

\noindent Let $x,y$
be an $M$-superficial sequence \wrt \ $I$. Set $J = (x,y)$ and $N
= M/xM$.

(a) By Proposition \ref{e2}(iii) we have that $\ov{\wt{I^jM}}
=  \wt{I^jN}$ for each $j \geq 1$. So by  \ref{h1zero}; we get $H^{1}(L^I(M)) = 0$.

(b) Using (a) and \ref{power-of-I}  we get $G_{I^n}(M)$ is \CM \ for all $n \gg 0$.

(c) From (b) and \ref{dag} (and corresponding long exact sequence in cohomology) it follows that
 $G_I(M)$ is generalized \CM. To
compute the $\mathbb{I}$-invariant of $G_I(M)$, note that from
Proposition \ref{e2}(ii) we get that
$\wt{I^{i+1}M } \subseteq I^iM$ for each $i \geq 1$.
Therefore note that
\begin{equation}\label{aug-a}
H^0_{\M}(G_I(M)) = \bigoplus_{i \geq 0}\frac{(\wt{I^{i+1}M}\cap I^iM)}{ I^{i+1}M} = \bigoplus_{i \geq 0}\frac{\wt{I^{i+1}M}}{ I^{i+1}M} = H^0(L^I(M)).
\end{equation}
So using the long exact sequence of cohomology corresponding to first fundamental exact sequence  and the fact that $H^{1}(L^I(M)) = 0$ we get
\begin{equation}\label{aug-b}
H^1_{\M}(G_I(M)) \cong  H^0(L^I(M))(-1).
\end{equation}
Therefore $\mathbb{I}(G_I(M)) = 2 \lambda\left( H^0(L^I(M)) \right) $.

Since $\wt{I^{i+1}M} = J\wt{I^iM}$ for all $i \geq 1$ it follows that
$\wt{G}_I(M)$ has minimal multiplicity. In particular $\wt{e}_3^I(M) = 0$. Using  \ref{rrhilbertcoeff}(b) we get
      $\lambda\left( H^0(L^I(M)) \right) = -e_{3}^{I}(M)$. The result follows.

(d) Since $\wt{I^{i+1}M } \subseteq
I^iM$ for each $i \geq 1$ we get that $G_I(A)_{+}$ kills $H^0(L^I(M))$.
Using \ref{aug-a} and \ref{aug-b} we get that $G_I(A)_+$ also annihilates  the  local cohomology modules $
H^{i}_{\M}(G_I(M))$ for $i = 0, 1$. In particular when $I = \m$ we get that
$G(M)$ is a quasi-Buchsbaum $G(A)$-module.

\noindent\textbf{ The case when } $M = A$, the ideal $A$ is integrally closed and
$e_{2}^{I}(A) = e_{1}^{I}(A) - e_{0}^{I}(A) + \lambda(A/I)$.

The proof of this can be given on exactly the same lines as the previous case (we have to use
Proposition \ref{chi2}). The only
thing to notice throughout is
 that $\wt{I} = I$ since $I$ is integrally closed.
\end{proof}

\section{ Ratliff-Rush filtration mod a
superficial sequence}
Let
$(A, \m)$ be a Noetherian local ring of dimension $d$ with infinite residue field $k$ and let $M$
is a \CM \ $A$-module of dimension $r \geq 2$.   Let $I$ be an ideal of definition
for $M.$ In this section we study vanishing of $H^{i}(L^{I}(M))$ for $1 \leq i \leq s$ where $s \leq r-1$ and
relate it to  behavior of Ratliff-Rush filtration mod superficial sequences (of length $\leq s$); see Theorem \ref{rrMODsup}.
The easiest case to handle is when $s =1$ and we took care of it in section 3.

\s \noindent\textbf{Notation}

\noindent Let $\bx  = x_1, \cdots, x_i \in I \setminus I^2$ be a
$M$-superficial sequence \wrt \ $I.$  We  assume $s \leq r -
1$ and as always $r \geq 2.$  Set
\[
M_0 = M, \quad \text{and} \quad   M_i = \frac{ M}{(x_1, \ldots , x_i)M} \quad \text{ for} \ 1 \leq i  \leq s.
\]

\s  For
 $0 \leq i < j \leq s$
 let $\rho^{i,j} \colon M_i
\rightarrow M_j$ be the natural map.  Clearly
$$\rho^{i,j} \left(\widetilde{I^n M_i} \right) \subseteq \widetilde{I^n M_j} \quad \text{ for all  $n \geq 1$}. $$
So $\rho^{i,j}$ induces
\[
\rho^{i,j}_n \colon  \frac{\wt{I^{n} M_i}}{I^n M_i} \longrightarrow \frac{\wt{I^{n} M_j}}{I^n M_j}  \quad \quad \quad   \text{for all} \ n \geq 1.
\]
Fix $n \geq 1$. It can be easily checked that  $\rho^{i,j} \left(\wt{I^{n} M_i} \right) = \left(\wt{I^{n} M_j} \right)$ \ff \ $\rho_n^{i,j}$ is \emph{surjective}.

\s \label{c-dia} Fix $n \geq 1$. For $0 \leq i < j < t \leq s$ we have a commutative diagram
\[
\xymatrixrowsep{3pc}
\xymatrixcolsep{2.5pc}
\xymatrix{
\wt{I^{n} M_i}/I^n M_i  \ar@{->}[rd]^{\rho^{i,t}_n}
     \ar@{->}[d]_{\rho^{i,j}_n}
&
    \
\\
\wt{I^{n} M_j}/ I^n M_j  \ar@{->}[r]_{\rho^{j,t}_n}
     &  \wt{I^{n} M_t}/I^n M_t
}
\]
It follows that if $ \rho^{i,j}_n$ and $\rho^{j,t}_n $ is surjective then $\rho^{i,t}_n$  is surjective.

\begin{definition}
We say the Ratliff-Rush filtration on $M$ \wrt \ $I$ \emph{behaves well mod }$\bx = x_1,\ldots,x_s$   if
$\rho^{0,s} \left(\wt{I^{n} M} \right)= \left(\wt{I^{n} M_s} \right)$ for all $n \geq 1$ (equivalently
$\rho^{0,s}_n$ is \emph{surjective} for all $n \geq 1$).
\end{definition}
Next we state the following generalization of Theorem \ref{h1zero}.
\begin{theorem}\label{rrMODsup}
Let $(A,\m)$ be local with an infinite residue field and let $M$ be a \CM \ $A$-module of dimension $r \geq 2$. Let $I$ be an ideal of definition for $M$.
Let $\bx = x_1,\ldots x_s$ be  $M$-superficial  \wrt \ $I$ and assume $s \leq r-1$. \TFAE
\begin{enumerate}[\rm (i)]
  \item The  Ratliff-Rush filtration on $M$ \wrt \ $I$ behaves well mod $\bx$.
  \item $H^{i}(L^{I}(M)) = 0$ for $i = 1,\ldots,s$.
\end{enumerate}
\end{theorem}

\begin{remark}\label{conseqGood}
If the  Ratliff-Rush filtration of $M$ \wrt\ $I$,  behaves well mod $\bx$ then $\widetilde{G}_I(M)/\bx \widetilde{G}_I(M) = \widetilde{G}_I(M_s) $.
 Since $\dim M_s \geq 1$ we get $ \depth \widetilde{G}_I(M_s) \geq 1$. So $ \depth \widetilde{G}_I(M) \geq s+ 1$, by  \cite[2.2]{Huck-Marley-97}.
It follows
that
\begin{enumerate}[\rm (a)]
  \item $\depth G_{I^n} (M) \geq s + 1$ for all $n \gg 0.$
  \item $\lambda (H^{i}(L^I(M)))$ is finite for $i = 0,\ldots,s$; see \ref{power-of-I}.
\end{enumerate}
\end{remark}

\begin{proof}[Proof of Theorem \ref{rrMODsup}]
 We
prove this result by induction on $s.$

$s = 1:$  This we proved in Theorem \ref{h1zero}.

$s \geq 2:$  We assume the result for $s - 1$ and prove it for
$s.$

If $H^i (L^I (M)) = 0$ for $i = 1,2, \cdots, s$ then
\begin{enumerate}[\rm (a)]
  \item by $s =
1$ case if follows that
$\rho^{0,1}_n$ is surjective
  for all $n \geq 1$.
  \item Using \ref{longH} it follows that $H^i (L^I (M_1))
= 0$ for $i = 1, \cdots, s - 1.$
\end{enumerate}
 By induction hypothesis it
follows that Ratliff-Rush filtration of $M_1$ behaves well mod $x_2,\cdots,x_s$. So
$\rho^{1,s}_n$ is surjective for all $n \geq 1$.
 Using (a) and  \ref{c-dia}  we get that the  Ratliff-Rush filtration on $M$ behaves well mod $\bx$.

   Conversely if the  Ratliff-Rush filtration on $M$ behaves well mod $\bx$ then by Remark \ref {conseqGood} we have
$\lambda (H^i (L^I (M))) < \infty$ for $i = 0, 1, \cdots, s.$ Using \ref{longH} inductively we can prove
\begin{equation*}
\lambda \left( H^i(L^I(M_j)) \right) < \infty \quad \text{for} \ i = 0,\ldots, s-j.  \tag{*}
\end{equation*}
We prove by downward induction on $j = s - 1, \ldots, 0$  that
\begin{equation*}
H^i (L^I (M_j)) = 0 \quad \text{for} \ i = 1, 2, \ldots, s - j.\tag{$\dagger$}
\end{equation*}
$j = s-1:$ \\
As $\rho^{0,s} = \rho^{s-1 , s} \circ \rho^{0, s - 1}$; for all
 $n \geq 1$ we have
  \begin{align*}
\widetilde{I^n M_s} = \rho^{0,s} (\widetilde{I^n M}) &=  \rho^{s - 1, s}\left(\rho^{0, s - 1} (\widetilde{I^n M}) \right) \\
&\subseteq \rho^{s - 1, s} (\widetilde{I^n M_{s - 1}}) \\
&\subseteq  \widetilde{I^n M_s}.
 \end{align*} Therefore     $ \rho^{s - 1, s} (\widetilde{I^n M_{s - 1}})  =  \widetilde{I^n M_s}$ for all $n \geq 1$.
Thus the Ratliff-Rush filtration of $M_{s-1}$ behaves well mod $x_{s}$. It follows from Theorem \ref{h1zero}
that $H^{1}(L^{I}(M_{s-1})) = 0$.

We assume the  result for $j = t$
and we prove to it for $j = t - 1.$

By hypothesis $H^i (L^I (M_t)) = 0$ for $i = 1, \ldots, s - t.$ Using \ref{Artin}(a), \ref{longH} and \ref{Artin-vanish} we get
that $H^i (L^I (M_{t - 1})) = 0$ for $i = 2, \ldots, s - t + 1.$

By  \ref{longH} we also have
\[
H^1 (L^I (M_{t - 1})) (-1) \xrightarrow{H^1 (\Psi_x)} H^1 (L^I
(M_{t - 1})) \rightarrow H^1 (L^I (M_{t})) = 0.
\]
 As $\lambda (H^1 (L^I (M_{t - 1}))) < \infty$ we get that $H^1(\Psi_x)$ is an isomorphism. So
 $$H^1 (L^I (M_{t - 1})) (-1)
\cong H^1 (L^I (M_{t - 1})).$$
  This implies that $H^1 (L^I (M_{t
- 1})) = 0$; (use \ref{Artin-vanish}).

Thus by downward induction it follows that $H^i (L^I
(M)) = 0$ for $i = 1, 2, \ldots, s.$
\end{proof}

    An easy consequence to the above theorem is
\begin{corollary}\label{invar-gen}[with hypothesis as in \ref{rrMODsup}]
Suppose $\bx= x_1,\ldots x_s \in I$ and $\by = y_1,\ldots,y_s \in I $ are two \textbf{distinct} $M$-superficial sequences \wrt \ $I$. Here $s \leq \dim M -1$. \TFAE
\begin{enumerate}[\rm (i)]
  \item The  Ratliff-Rush filtration on $M$ \wrt \ $I$ behaves well mod $\bx$.
  \item The  Ratliff-Rush filtration on $M$ \wrt \ $I$ behaves well mod $\by$.
  \end{enumerate}
\qed
\end{corollary}

%55555555555555555555555555555555555555555555555555555555555555555555555555
%666666666666666666666666666666666666666666666666666666666666666666666666
%777777777777777777777777777777777777777777777777777777777777777777777777

%000000000000000000000000000000000000000000000000000000000000

\section{Finite  local cohomology and minimal $\I$-invariant}
We relate the finite generation of local  cohomologies of
$G_I(M)$ and $L^I (M).$ Set
\begin{align*}
 \fg_{I} (M) &= \max \{j \mid \lambda (H^i (G_I (M))) < \infty \quad
\mbox{for}  \  i = 0,1,\cdots,j-1\} \quad \text{and} \\
\alpha_I (M) &= \max \{j \mid j \leq r-1 \ \ \text{and} \ \lambda (H^i (L^I (M))) < \infty  \ \mbox{for} \  \  i = 0,1,\cdots, j-1\}.
\end{align*}
  Using \ref{dag} and the corresponding the long exact sequence in cohomology we get
$\fg_I (M) \geq \alpha_I (M).$  We prove that
$$\fg_I(M) = \alpha_I(M).$$
An application of our result
is the notion of generalized \CM \ modules with $\depth $ zero having minimal $\I$-invariant. This notion
also relates to Ratliff-Rush filtration on $M$ behaving well mod superficial sequences.

Let $r = \dim M.$
Clearly $0 \leq \fg_I (M) \leq r - 1.$
The following proposition is easy to prove.  So we omit the proof.
\begin{proposition}\label{fg-mod}
Let $M$ be a $CM$ $A$-module of dimension
$r$ and $I$ is an ideal of definition for $M$.  Let $\bx = x_1,
\cdots, x_s$ be an $M$-superficial sequence  \wrt \ $I.$  Then
\begin{equation*}
\fg_I (M/\bx M) \geq  \fg_I (M) - s \quad \text{and} \quad        \alpha_I (M/\bx M) \geq \alpha_I (M) - s. \end{equation*}
\qed
\end{proposition}

We now prove the main theorem of this section.
\begin{theorem}\label{alpha=fg}
$\fg_I (M) = \alpha_I (M).$
\end{theorem}
\begin{proof}
 We first prove $\fg_I (M) \geq s \Rightarrow
\alpha_I (M) \geq s$ by induction on $s.$   There is nothing to show when $s = 1$.
For convenience of the reader we will explicitly write out the
proof in the case $s = 2.$  We claim that in this case
 $\lambda(H^1 (L^I (M))) < \infty.$
  As $\lambda (H^1 (G_I (M))) < \infty$ it follow that the map
\begin{equation*}
H^1(\Pi)_n : H^1 (L^I (M))_n \rightarrow H^1 (L^I (M))_{(n - 1)} \quad
\;\mbox{is injective for} \ n \ll 0.
\end{equation*}
Set $N = M/xM$ where $x \in I\setminus I^2$ is $M$-superficial \wrt\  $I$.
Using  \ref{longH} and since $H^0 (L^I (N))$ has finite length
we get that
\begin{equation*}
H^1 (\Psi_x)_n : H^1 (L^I (M))_{(n - 1)} \rightarrow H^1 (L^I
(M))_n \;\mbox{is injective for} \ n < 0.
\end{equation*}
 Using \ref{deg-vanish} we get $H^1 (L^I (M))_n = 0$
for all $n \ll 0.$  The result follows.\\

Suppose $\fg_I (M) = s \geq 2.$  Let $x \in I\setminus I^2$ be
$M$superficial \;\wrt\; $I.$  Set $N = M/xM.$
Then $\fg_I (N) \geq s - 1.$ by \ref{fg-mod}.
By induction hypotheses $\alpha_I (N) \geq s - 1.$  We claim
$\lambda (H^i (L^I (M))) < \infty$ for $i = 0, 1, \cdots, s.$
\begin{equation*}
\lambda (H^0 (L^I (M))) < \infty \  \mbox{by \ref{zero-lc}}.
\end{equation*}
Since $\lambda (H^i (L^I (N))) < \infty$ for $i = 0, 1, \cdots, s
- 1$ it follows from \ref{longH} that the map
\begin{equation*}
H^1 (\Psi_x)_n : H^i (L^I (M))_{(n - 1)} \rightarrow H^i (L^I
(M))_n
\end{equation*}
is injective for all $n \ll 0.$  and all $i = 1, 2, \cdots, s.$\\
Also since $\lambda(H^i (G)) < \infty$ for $i = 0, 1, \cdots, s$ it
follows from the first fundamental exact sequence, \ref{dag}, that
\begin{equation*}
H^1 (\Pi)_n : H^i (L^I (M))_n \rightarrow H^i (L^I (M))_{(n - 1)}
\;\mbox{is injective}
\end{equation*}
for all $n \ll 0$ and all $i = 1, 2, \cdots, s$.

Using \ref{deg-vanish} we get
\begin{equation*}
H^i (L^I (M))_n = 0 \;\mbox{for all }\; n \ll 0.
\end{equation*}
Thus $\lambda (H^i (L^I (M))) < \infty$ for $i = 1, 2, \cdots,
s.$\\
The result follows.
\end{proof}
%%%%%%%%%%%%%%%%%%%%%%%%%%%%
\s \textbf{Minimal $\I$-invariant:}

Recall that we say $G_I (M)$ is generalized \CM \ module if
\begin{equation*}
\lambda (H^i (G_I (M))) < \infty \;\mbox{for}\; i = 0, 1,\cdots, r
- 1.
\end{equation*}
For generalized \CM \ module the St\"{u}ckrad-Vogel invariant
\begin{equation*}
\I(G_I (M)) = \sum^{r - 1}_{i = 0} \binom{r - 1}{i}
\lambda (H^i (G_I (M)))
\end{equation*}
plays a crucial role.\\
If $x^* \in G_I (A)_1$ is $G_I (M)$- regular and $N = M/xM$ then
we can easily verify that
\begin{equation*}
\I(G_I (N)) \leq  \I(G_I (M)).
\end{equation*}
So in some sense if we have to study minimal $\I$- invariant then
we have to first consider the case when  $ \depth G_I (M) = 0.$  We
prove
\begin{theorem} \label{vogel}
Let $M$ be a \CM \  $A$-module of dimension $r \geq 2.$  Assume $G_I
(M)$ is generalized $CM$ and  $ \depth G_I (M) = 0.$  Then
\begin{equation*}
\I(G_I (M)) \geq r \bullet \lambda \left( H^0 (G_I (M)) \right).
\end{equation*}
Furthermore the following are equivalent
\begin{enumerate}[\rm (i)]
\item
$\I (G_I (M)) = r \bullet \lambda (H^0 (G_I (M)))$
\item
$H^i (L^I (M)) = 0$ for $i = 1, 2, \cdots, r - 1.$
\item
 The Ratliff-Rush
filtration on $M$ behaves well mod superficial sequences.
\end{enumerate}
\end{theorem}
%\begin{proof}
\begin{proof}
By Theorem \ref{alpha=fg} it follows that $\lambda (H^i (L^I (M)))$ is finite
for $i = 0, 1, \cdots, r - 1$.  Using \ref{dag}  and the corresponding long exact sequence in  local cohomologies we get   an exact sequence
\begin{align*}
0 &\rightarrow H^0(G_I (M)) \rightarrow H^0(L^I(M)) \rightarrow H^0 (L^I (M))(-1)  \\
&\rightarrow H^1 (G_I (M)) \rightarrow H^1 (L^I(M)) \xrightarrow{H^1(\Pi)} H^1 (L^I (M)) (-1)
\end{align*}
Let $K = \ker H^1 (\Pi).$  It follows that
\begin{equation*}
\lambda (H^1 (G_I (M))) = \lambda (H^0 (G_I (M))) + \lambda (K)
\geq \lambda (H^0 (G_I (M)))
\end{equation*}
Thus
$$\I (G_I (M)) \geq \lambda (H^0 (G_I (M))) + \binom{r -1}{1} \lambda (H^0 (G_I (M))) = r \lambda (H^0 (G_I (M))).$$
Furthermore equality holds \ff \ $\lambda (H^i (G_I (M))) = 0$ for
$i = 2, \cdots, r - 1$ and $K = 0.$\\
Its clear that this holds \ff \ $H^i (L^I (M)) = 0$ for $1 \leq i
\leq r - 1.$  Thus $(1) \Leftrightarrow (2)$ holds.

  The
equivalence $(2) \Leftrightarrow (3)$ follows from Theorem \ref{rrMODsup}.
\end{proof}

 \section{Generalization of a result due to Narita.}

 A classical result, due to Narita \cite{Nar}  states that if
  $(A, \m)$ is \CM \  of $\dim 2$ then
 \begin{equation*}
 e^I_2 (A) = 0 \ \ \mbox{iff}\ \  \red (I^n) = 1 \;\mbox{for all} \ n \gg 0.
 \end{equation*}
 This result can be easily extended to Cohen-Macaulay modules of
 dimension two.  However  Narita's result fails (even for \CM \
 rings) in dimension $\geq 3$; see \ref{CPR}.

  To motivate our generalization we
 will first reformulate Narita's result in dimension two. Our reformulation is
 \begin{theorem}\label{reformul}
 Let $M$ be $CM$ of dimension two.  Then \tFAE
 \begin{enumerate}[\rm (i)]
 \item
  $e^I_2 (M) = 0$.
 \item
  $\tilde{G_I} (M)$; the associated graded module of the Ratliff-Rush filtration has minimal multiplicity.
 \end{enumerate}
 \end{theorem}
 Recall that we say graded $G_I(A)$-module $E = \bigoplus_{n\geq 0} E_n$
 has minimal multiplicity if it is \CM \ and $\deg h_E (z) \leq 1.$

 We give
 \begin{proof}[Proof of  Theorem \ref{reformul}]
 The assertion $(2) \Rightarrow (1)$
 is clear.  Conversely note that by \ref{e2} we get  that
 \begin{equation*}
 \text{if} \ \  e^I_2 (M) = 0 \;\mbox{then}\; \widetilde{I^{j + 1} M} = J
 \widetilde{I^j M} \;\mbox{for}\; j \geq 1.
 \end{equation*}
 By \cite[1.2]{RV-big} it follows that $\widetilde{G}_I (M)$ has
 minimal multiplicity
 \end{proof}
  Theorem \ref{reformul} enables us to
 generalize  Narita's result as follows:

 \begin{theorem}\label{genNar}
 Let $M$ be $CM$ of dimension $r \geq 2.$  Then \tFAE
 \begin{enumerate}[\rm (i)]
 \item
  $e^I_2 (M) = \cdots = e^I_r (M) = 0$.
 \item
  $\tilde{G}_I(M)$; the associated graded module of the Ratliff-Rush filtration has minimal multiplicity.
 \end{enumerate}
 Furthermore if these condition hold then
 \begin{enumerate}[\rm (a)]
 \item
 $H^i(L^I(M)) = 0$ for $i = 1,\ldots, r-1$.
 \item
  $G_I(A)$ is generalized \CM \ with
 \[
 \I(G_I(M)) = (-1)^{r+1}r\bullet e^I_{r+1}(M)
 \]
 \end{enumerate}
 \end{theorem}
 \begin{proof}
 (ii) $\implies $ (i).   This is clear.

 (i) $\implies $ (ii)
 We prove following assertion by induction on $r$ where
 $r \geq 2$.
 \[
 e^I_i (M) = 0 \ \text{for} \ \ i = 2,\cdots, r \Rightarrow
 \begin{cases}
 (a)  \ H^i (L^I (M)) = 0 \ \text{for}\ i = 1, 2, \cdots, r - 1\\
 (b) \ \tilde{G}_I (M) \ \text{ has minimal multiplicity.}
 \end{cases}
 \]
 For $r = 2,$ this follows from Theorem \ref{e2dim2main}.
 We assume $r \geq 3$ and that the result holds for $r - 1.$

  Let
 $x \in I$ be $M$-superficial \wrt \  $I.$  Set $N = M/xM.$
 By \ref{mod-sup-h}(3) we get $e_i (N) = 0$ for $i = 2,\cdots, r - 1.$  So by
 induction hypothesis; we get $H^i (L^I (N)) = 0$ for $i = 1, 2,\ldots, r - 2$ and $\widetilde{G}_I(N)$ has minimal
 multiplicity.

 By \ref{longH} and \ref{Artin-vanish} we get  $H^i (L^I (M)) = 0$ for $i = 2, \cdots, r - 1.$
 Let $K = \ker (\Psi_x : H^1 (L^I (M))) (-1) \rightarrow H^1 (L^I
 (M)).$
  We use  the following part of \ref{longH}
 \begin{equation*}
 0 \rightarrow \B \rightarrow H^0 (L^I (M)) (-1) \rightarrow H^0
 (L^I (M)) \rightarrow H^0 (L^I (N)) \rightarrow K \rightarrow 0 \tag{$\dagger$}
 \end{equation*}
 Notice $\dim N = r-1 \geq 2$. By \ref{rrhilbertcoeff}  and the fact  $\widetilde{G}_I(N)$ has minimal
 multiplicity  we get
 \begin{equation*}
 e_{r}^{I}(N) =  (-1)^r \sum_{r \geq 1} \lambda
 \left(\frac{\widetilde{I^j N}}{I^j N}\right) = (-1)^r \lambda (H^0
 (L^I (N))) \tag{*}
 \end{equation*}
 By \ref{mod-sup-h}(3) we also have
 \begin{equation*}
 e_{r}^{I}(M) = e_{r}^{I}(N) - (-1)^r \lambda (\B).
 \end{equation*}
 Since $e_{r}^{I}(M) = 0$. we get
 \begin{equation*}
 \lambda (B) = (-1)^r e_r (N). \tag{**}
 \end{equation*}
 By computing lengths from ($\dagger$) and using (**) we get
 \begin{equation*}
 (-1)^r e_r (N) - \lambda (H^0 (L^I (N))) + \lambda (K) = 0.
 \end{equation*}
 Using (*) it follows that $K= 0$; i.e., the natural map $H^0(L^I(M)) \rt H^0(L^I(N))$ is surjective. By Theorem \ref{h1zero} we get
 $H^1(L^I(M)) = 0$.
 Thus we have proved $H^i(L^I(M)) = 0$ for $1 \leq i \leq r -
 1.$

 Since $H^1(L^I(M)) = 0$ the Ratliff-Rush filtration on $M$ \wrt \ $I$ behaves well mod a superficial element.
 So
 \begin{equation*}
 \frac{\widetilde{G}_I (M)}{x^{\ast} \widetilde{G}_I (M)} \cong
 \widetilde{G}_I (N).
 \end{equation*}
 Thus $\wt{G}_I(M) $   is \CM \ and $\widetilde{h}_N (z) = \widetilde{h}_M (z).$  It follows that
 $\widetilde{G}_I(M)$ has minimal multiplicity.

 We now prove the other assertions.\\
 (a): This was proved while showing (ii) implies (i). \\
 (b): Using the  first fundamental exact sequence \ref{dag} and the corresponding long exact sequence in cohomology we get
 $H^{i}(G_I(M)) = 0$ for $i \geq 2$. Furthermore we have an exact sequence
 \[
 0 \rt H^0(G_I(M)) \rt H^0(L^I(M)) \rt H^0(L^I(M))(-1) \rt H^1(G_I(M)) \rt 0
 \]
 Since $\wt{G}_I(M)$ has minimal multiplicity it follows that
 $\wt{I^{n}}M \subseteq JM$ for all $n \geq 1$; \cite[1.9]{GuRossi}. It follows that
 \[
 H^0(G_I(M)) \cong H^0(L^I(M)). \quad \text{So} \  H^0(L^I(M))(-1) \cong H^1(G_I(M)) \ \  \text{also}.
 \]
 It follows that $\I(G_I(M)) = r \lambda(H^0(L^I(M)))$. The result follows by 1.8(b).
 \end{proof}

 Let $I$ be an ideal which satisfies generalized Narita's theorem.  It is possible to have $\depth G_I(A) = 0$ as shown by the following example
 taken from \cite[3.8]{Cpr} shows
 \begin{example}  \label{CPR}
 Let $A = \mathbb{Q}[[x,y,z]]$. Let $I= (x^2-y^2, y^2-z^2, xy,xz,yz)$. Set $\m = (x,y,z)$. It can be verified that
 $I^2 =  \m^4$. So $\red(I^n) = 2$ for all $n \gg 0$.
 Using COCOA  we get
 $h_I(A,z) =  5 + 6z^2 - 4z^3+z^4$. So $e_j^I(A) = 0$ for $j = 2,3$.
 In \cite{Cpr} it is  proved that $\depth G_I(A) = 0$.
 \end{example}
%%%%%%%%%%%%%%%%%%%%%%%%%%%%%%%%%%%%%%%%%%%%%%%%%%%%%%%%%%%%%%%%%%%%%%%%%%%
Recall that $$\xi_I(M) = \lim_{n\to\infty} \depth G_{I^n}(M).$$
When $e_2^I(M) = 0$ and dimension $M = 3$ we prove that either $\fg_I(M) = 1$ or $= 3$.
This example will illustrate some of the techniques developed in this paper.
\begin{example}\label{e2dim3} Let $A$ be local and let
$M$ be  a finite Cohen-Macaulay $A$-module of dimension $ r = 3$. Let $I$ be an
ideal of definition for $M$.
 Assume
$e_{2}^{I}(M) = 0$.
then $e_{3}^{I}(M) \leq 0$.
Furthermore we have
\begin{enumerate}[\rm (a)]
\item
\TFAE
\begin{enumerate}[\rm (i)]
\item
$e_{3}^{I}(M) < 0$
\item
 $ \xi_I(M) = 1 $.
 \item
$\fg_I(M) = 1$.
\end{enumerate}
\item
\TFAE
\begin{enumerate}[\rm (i)]
\item
$e_{3}^{I}(M) = 0$
\item
$\xi_(M)  = 3$.
\item
 $\fg_I(M) = 3$.
 \end{enumerate}
\end{enumerate}
\end{example}
\begin{proof}
We first prove that $e_3^I(M) \leq 0$.

Let $x$ be $M$-superficial \wrt \ $I$. Set $N = M/xM$. Then by
Theorem \ref{e2dim2main} we have $H^1(L^I(N)) = 0$. Since $H^1(L^I(N)) = 0$ and $H^2(L^I(M))$ is *-Artinian,  using
(\ref{longH})
we get
\begin{equation*}
    H^2(L^I(M)) = 0. \tag{$\dagger$}
\end{equation*}
 We also use
 the following part of (\ref{longH})
\begin{equation*}
0 \xar \B^{I}(x, M) \xar H^0(L^I(M))(-1) \xar   H^0(L^I(M))  \xar H^0(L^I(N)).
\end{equation*}
Therefore we have
\begin{equation}
\label{p2e1uu}
b :=  \lambda\left(\B^{I}(x,M)\right) \leq  \lambda\left(H^0(L^I(N))\right) =: r.
\end{equation}
Since $e_{i}^{I}(N) = e_{i}^{I}(M)$ for $i = 0,1,2$, we get by
using Propositions \ref{e2}.(i), \ref{chi2}.(i) and \ref{sigma}
that $\wt{e_{3}}^{I}(N) = 0$. Using \ref{rrhilbertcoeff} we get that
$e_{3}^{I}(N) = -r$.
So we have
\begin{equation*}
 b \leq - e_{3}^{I}(N). \tag{*}
\end{equation*}
By \ref{mod-sup-h}.3(b) we get that
$e_{3}^{I}(M)= e_{3}^{I}(N) + b$.  Using (*) we get
$e_{3}^{I}(M) \leq 0$ with equality if and only if $b = r$.

Notice if $b = r$ then the map $ H^0(L^I(M))  \rt H^0(L^I(N)) $ is surjective.
Using \ref{h1zero} we get $H^1(L^I(M)) = 0$. Since $H^0(L^I(M))_{-1} = 0$ it follows
from \ref{power-of-I} that $\xi_I(M) = 3$.

We also note that
\begin{equation}\label{e3st}
e_3^{I^n}(M) = e_3^I(M) \quad \text{for all} \ n \gg 0.
\end{equation}

We now prove our assertions \\ (b) (i) $\implies$ (ii) \\
This follows from the fact that $e_3^I(M) = 0$ \ff \ $b = r$. We have shown above that if $b = r$ then
$\xi_I(M) = 3$.\\
(b) (ii) $\implies$ (i) \\ If $\xi_I(M) = 3$; say $G_{I^r}(M)$ is \CM. Then clearly all Hilbert coefficients
of $M$ is non-negative. Using \ref{e3st} we get $e_3^I(M) \geq 0$. But we have shown earlier that in general
if $e_2^I(M) = 0$ then $e_3^I(M) \leq 0$. So
 we have $e_3^I(M) = 0$.

 We now prove   \\
 (a) (i) $\implies$ (ii) \\ If $e_3^I(M) < 0$ then it follows from \ref{e3st} and \cite[Corollary 2]{Marley-89} that $\xi_I(M) = 1$.\\
 (a) (ii) $\implies$ (i) \\  This follows from
 (b) (i) $\implies$ (ii) and the fact that $e_3^I(M) \leq 0$ always when $e_2^I(M) = 0$.

 Before proving the remaining assertions let us note the following
  \[
    3 \geq \fg_I(M) \geq \xi_I(M) \geq 1
  \]
  So
 (b)(ii) $\implies$ (iii)  and
 (a) (iii) $\implies$ (i) \\

 To prove the   rest note that it suffices to show
 that if $e_2^I(M) = 0$ then
 \[
  \fg_I(M) \geq 2 \implies   \xi_I(M) = 3.
 \]
 If $\fg_I(M) \geq 2$ then note that $\lambda(H^1(L^I(M))) $ is finite.  Since
 $H^2(L^I(M)) = 0$, we get from \ref{longH}   an exact sequence
 \[
   H^0(L^I(N)) \rt H^1(L^I(M))(-1) \rt H^1(L^I(M)) \rt 0
 \]
   Therefore    $H^1(L^I(M)) = 0$.       The result follows from \ref{power-of-I}.
\end{proof}
%666666666666666666666666666666666666666666666666
%77777777777777777777777777777777777777777777777777777777777777777777777777777777777777

%%%%%%%%%%%%%%%%%%%%%%%%%%%%%%%%%%%%%%%%%%%%%%%%%%%%%%%%%%%%%%%%%%%%%%%%%%%
\section{Asymptotic invariants.}
Throughout $M$ is \CM \ $A$-module of dimension $r$ and $I$ is an ideal of definition for $M$.
Suppose $\fg_I (M) = s.$  We compute  $H^i (G_{I^n} (M))$ for $n \gg 0$ and $0\leq i \leq s-1$.
The basic idea is to use is that $L^I(M)(-1)$ behaves well with respect to the Veronese functor. We have
\[
\left(L^I(M)(-1) \right)^{<l>} = L^{I^l}(M)(-1) \quad \text{for  } \ l \geq 1.
\]
We also prove that if $x$ is $M$-superficial \wrt \ $I$ then
\[
  \xi_I(M/x^sM) \geq \xi_I(M)-1 \quad \text{for all} \ s \gg 0.
\]

\begin{theorem}\label{asymp}
Let $(A,\m)$ be Noetherian local and
let $M$ be a $CM$ $A$-module of dimension $r \geq 1.$  Let $I$ be an ideal for definition of $M$.  Let $\fg_I
(M) = s.$  Then for $n \gg 0$ we have
\begin{enumerate}[\rm (a)]
\item
$H^0 (G_{I^n} (M)) = 0$.
\item
 $$ H^1 (G_{I^n} (M))_j =\left\{{\begin{array}{rll}
0 & \mbox{for} & j \neq -1\\
H^1 (L^I (M))_{-1} & \mbox{for}  & j = -1\end{array}}\right. %}
$$
\item
 For $2\leq i \leq s-1$ we have
$$ H^i (G_{I^n} (M))_j =\left\{{\begin{array}{rlll}
0 & \mbox{for} & j \neq 0, -1\\
H^i (L^I (M))_{-1} & \mbox{for}  & j = -1\\
H^{i - 1} (L^I (M))_{-1} & \mbox{for} & j = 0.\end{array}}\right. %}
$$
\end{enumerate}
\end{theorem}
\begin{proof}
By Theorem 5 it follows that $\alpha_I (M) = s.$  Fix $i$ with $0 \leq i \leq s-1$.

 Since $H^i
(L^{I^n} (M) (-1)) \simeq H^i (L^{1} (M) (-1))^{<n>}$ for all $n$;
it follows that for $n \gg 0$ we have
\begin{equation*}
H^i (L^{I^n} (M))_j = 0 \;\mbox{for}\; j \neq -1. \tag{*}
\end{equation*}
Assume that for $n \geq n_0$ the assertion $(*)$ holds for all $i = 0, 1,\ldots, s-1$ and that $H^0 (L^{I^n} (M)) =
0.$
Fix $n \geq n_0.$  Set $K = I^n.$

 (a) This holds by construction.\\

 (b) Set $G = G_K (M)$ and $L = L^K (M)$
By the first fundamental sequence  \ref{dag}  and the corresponding long exact sequence in cohomology we obtain for all $j \in \Z$
\begin{equation*}
0 \rightarrow H^1 (G)_j \rightarrow H^1 (L)_{j-1} \rightarrow H^1 (L)_j \rightarrow H^2
(G)_j \rightarrow H^2 (L)_j \rightarrow H^2 (L)_{j-1} \tag{$\dagger$}
\end{equation*}
 Using the above exact sequence we get (b).\\
 (c)
Also note that $(\dagger)$ also implies
\begin{align*}
H^2 (G)_j &= 0 \quad \text{for} \ j\neq 0, -1 \\
H^2 (G)_0 \ &\cong H^1 (L)_{-1} \\
H^2 (G)_{-1} \  &\cong H^2(L)_{-1}.
\end{align*}
 Thus we have proved (c) when $i = 2$. For $i \geq 3$ the proof is similar to the case $i = 2$.
\end{proof}

\begin{definition} We call the
 $A$-modules $H^i (L^I (M))_{-1}$ for $i = 1,\cdots,r - 1$
the asymptotic invariants of $G_I (M)$
\end{definition}

\s Recall $\xi_I (M) = \depth G_{I^n} (M)$ for all $n \gg 0.$ Notice that
$\xi_{I^r}(M) = \xi_I(M)$ for any $r \geq 1$.

Let $x \in I$ be $M$-superficial \wrt \ $I$.
One of question that we want to answer is whether
\begin{equation*}
\xi_I (M/xM) \geq \xi_I (M) - 1 ?
\end{equation*}
Although we have not been able to answer the question above in
general we prove

\begin{theorem}\label{xi-mod-x-large}
$\xi_I (M/x^n M) \geq \xi_I (M) - 1$ for all $n \gg 0$.
\end{theorem}
\begin{proof}
Notice $\fg_I(M) \geq \xi_I(M)$.
Assume for all $n \geq n_0,$
we have $H^i(L^{I^n}(M))_n = 0$ for all $n \neq -1$ and for $i = 0,\ldots, \fg_I(M) -1$.
   Fix $n \geq n_0.$

Set $L = L^{I^n} (M) , \bar{L} = L^{I^n} (M/x^n M).$  By \ref{longH}
 we have that for all $i \geq 0$ and $j \in \Z$;
\begin{equation*}
H^i (L)_{(j - 1)} \rightarrow H^i (L)_{j} \rightarrow H^i
(\bar{L})_{j} \rightarrow H^{i + 1} (L)_{j - 1} \rightarrow
H^{i+1} (L)_{j}
\end{equation*}
Set $c = \xi_I (M).$
Then $H^i (\bar{L})_{- 1} = 0$ for $i < c-1.$
 Clearly $\fg_I (M/x^n M) \geq \fg_I (M) - 1 \geq c -1.$

It follows from \ref{power-of-I} that
$$\xi_{I^n} (M/x^n M) \geq c -1.$$
Since $\xi_I (M/x^n M) = \xi_{I^n} (M/x^n M)$ the result follows.
\end{proof}

The following example shows that strict inequality can occur in Theorem \ref{xi-mod-x-large}.

\begin{example}
Let $(B,\n)$ be a two dimensional $CM$ local ring with
 an
$\n$-primary ideal $J$ such that $e_2(J) = 0$ and $G_J(B)$ has depth zero.
It can be easily checked that this is equivalent to $\tilde{J} \neq J$.
For a specific example of this kind we use an example from \cite{Marley-89}:
\[
  R = K[X,Y], \quad I = (X^7,X^6Y, XY^6, Y^7).
\]
 It is proved in \cite[page 8]{Marley-89} that $\depth G_I(A) = 0$ and that $e_2^I(A) = 0$. It is easily shown
 that this implies $\wt{I} \neq I$.

Notice by \ref{e2dim2main} that $H^1 (L^J(B)) = 0$ and  $e_3^I(B) < 0$.

 Set $A = B [X]_{(\n,x)}$ and $I =
(J,X).$
Then $G_I (A) \cong G_J (B) [X^*]$. We claim

 \begin{enumerate}[\rm (1)]
   \item
    $\xi_I(A) = 1.$
    \item
    $\fg_I(A) = 1$.
   \item
   $\xi_I(A/(X^n)) \geq 2$ for all $n \geq 1$.
 \end{enumerate}
 \emph{Proof of Claim:} \\
Notice
$e_3^I(A) = e^3_I(B) < 0$. So (1) and (2) follow from \ref{e2dim3}(a).

 (3). Since $H^0(L^I(B))_n = 0$ for $n< 0$ and $H^1(L^I(B)) =0$; by \ref{longH} we get  $\lambda(H^1 (L^I (A))_n) =
\lambda(H^1 (L^I (A))_{-1})$  for all $n < 0.$  So

   Fix $n
\geq 1$  Set $K = I^n$.  Notice
\begin{equation*}
H^0 (L^K (A)) (-1) = H^2 (L^K (A)) = 0.
\end{equation*}
 Set $N = A/(x^n)$. Using \ref{longH} we get
\begin{equation*}
0 \rightarrow H^0 (L^K (N))_n \rightarrow H^1 (L^K (A))_{n - 1}
\rightarrow H^1 (L^K (A))_n \rightarrow H^1 (L^K (N))_n
\rightarrow 0
\end{equation*}
At any rate $H^0 (L^K (N))_j = 0$ for $j < 0.$
Also $H^1 (L^K (A))_j \cong H^1 (L^K (A))_{-1} \forall j < 0.$
It follows that $H^{1}(L^K(N))_j = 0$ for $j < 0$.
Therefore $\xi_K (A/(X^n)) \geq 2.$  It follows that $\xi_I (A/(X^n))
\geq 2.$
\end{example}
%%%%%%%%%%%%%%%%%%%%%%%%%%%%%%%%%%%%%%%%%%%%%%%%%%%%%%%%%%%%%%%%%%%%%%%%%%%
\section{$\m$-primary ideals  with reduction number $2$}

               To give bounds on $\xi_I(M)$ in terms of Hilbert coefficients is in general a difficult task. Surprisingly the   following holds
\begin{theorem}\label{red2}
Let $M$ be a \CM \ module of dimension $3$ and let $I$ be an ideal
of definition for $M$ with $\red_I (M) = 2.$  Then $e^I_3 (M) \leq
0.$ Furthermore
 $$e^I_3 (M) = 0    \text{ \ff} \ \xi_I(M) \geq 2.$$
\end{theorem}
\begin{proof}
We first prove $e_3^I(M) \leq 0$.
We may choose $n_0$-such that for all $n \geq n_0$
\begin{align*}
H^0 (L^{I^n}(M)) &= 0 \\
H^i (L^{I^n} (M))_j &= 0 \quad \text{for} \ j \geq 0.
\end{align*}
        This is so since $L^I(M)(-1) = \bigoplus_{n\geq 0} M/I^nM$. \\
By \cite[2.4]{Hoa-93} we may choose $n_1$ such that for all $n \geq n_1$
\begin{equation*}
H^i (G_{I^n} (M))_j = 0\;\mbox{for}\; j \geq 1 \;\mbox{and}\; i =
0,1,2
\end{equation*}
and since $\red_I (M) = 2$ we get by \cite[3.2]{Trung-87} and    \cite[2.4]{Hoa-93} that
\begin{equation*}
H^3 (G_{I^n} (M))_j = 0 \;\mbox{for}\; j \geq 0.
\end{equation*}

Let  $\phi_M (I,z)$ be the shifted Hilbert-Samuel polynomial of
$M$ \;\wrt\; $I$\\
i.e., $\phi_M (I, n) = \lambda \left(M/I^n M\right)$ for $n \gg 0.$
\begin{equation*}
\phi_M (I,z) = e^I_0 (M)\binom{z+2}{3}  - e_1^I(M)\binom{z+1}{2} +
e_2^I(M)\binom{z}{1} - e_3^I(M).
\end{equation*}
Let $n_2 =$ postulation number of $\phi_M (I,z)$ i.e.
\begin{equation*}
\phi_M (I,n) = \lambda \left(M/I^n M\right)\;\mbox{for}\; n \geq
n_2.
\end{equation*}

\begin{equation*}
\mbox{Fix}\; r \geq \max \{n_0, n_1, n_2\}.
\end{equation*}
Set $K = I^r$.

 \emph{Claim 1:}  $H^2 (L^K (M))_{-1} = 0.$\\

 Set $G = G_K (M), L = L^K (M)$.
 Using the first fundamental exact
sequence  and the corresponding long exact sequence in cohomology we get
\begin{equation*}
H^2 (G)_j \rightarrow H^2 (L)_j \rightarrow H^2 (L)_{j - 1}
\rightarrow H^3 (G)_j \tag{*}
\end{equation*}

As $H^2 (L)_0 = H^3 (G)_0 = 0$ we obtain $H^2 (L)_{-1} = 0.$

\begin{equation*}
\mbox{Set}\ h_i = \lambda \left(H^i (G)_0\right)
\;\mbox{for}\; i = 0, 1, 2, 3.
\end{equation*}
Then $h_0 = h_3 = 0.$

\emph{Claim 2: }   $h_1 = 0$ and $h_2 = \lambda(H^1(L)_1)$.\\
Using the first fundamental exact
sequence  and the corresponding long exact sequence in cohomology we get
\begin{equation*}
H^0 (L)_{n - 1} \rightarrow H^1 (G)_n \rightarrow H^1 (L)_n
\rightarrow H^1 (L)_{n - 1} \rightarrow H^2 (G)_n \rightarrow H^2
(L)_n
\end{equation*}
Therefore $h_1 = 0$ and $h_2 = \lambda (H^1 (L)_{-1})$.\\

\emph{Claim 3: }      $e^I_3 (M) = -\lambda (H^1 (L^I (M))_{-1}) $.

Let
$$P^K_M (Z) = c_0 \binom{z+2}{2} - c_1 \binom{c+1}{1} + c_2$$
 be the Hilbert polynomial of $G_K(M)$ i.e.
\begin{equation*}
P^K_M (n) = \lambda \left(K^{n} M/K^{n + 1} M\right)\;\mbox{for}\;
n \gg 0.
\end{equation*}
By Grothendieck-Serre  formula \cite[4.4.3]{BH}  we get
\begin{equation*}
H^K (M,n) - P^K_M (n) = \sum^{3}_{i = 0} (-1)^2 \lambda (H^i (G_K
(M))_n)
\end{equation*}
Sot for $n = 0$ we get
\begin{equation*}
\lambda \left(M/I^r M\right) - \left[c_0 - c_1 +
c_2\right] = h_0 - h_1 + h_2 - h_3 = h_2. \tag{$\dagger$}
\end{equation*}

Write
\begin{equation*}
\phi_M (I^r,z) = c_0\binom{z+2}{3}  - c_1\binom{z+1}{2} + c_2\binom{z}{1} - c_3.
\end{equation*}
Clearly $\phi_M (I^r,z) = \phi_M (I, rz).$  In particular $c_3 =
e^I_3 (M)$\\
Also notice that
\begin{equation*}
\phi_M (I^r,1) = c_0 - c_1 + c_2 - c_3 = \phi_M (I,r) = \lambda
\left(M/I^r M\right)
\end{equation*}
(the last equality holds since $r \geq n_2$).\\
So by $(\dagger)$ we get
\begin{equation*}
h_2 = - c_3 = - e^I_3 (M).
\end{equation*}
Thus $e^I_3 (M) = -\lambda (H^1 (L^I (M))_{-1}) \leq 0.$

By Proposition
9.2  in Part 1 it follows that $e^I_3 (M) = 0$ \ff \\  $ \depth G_{I^n} (M) \geq
2.$

Otherwise note that    $H^1 (L^I (M))_{-1} \neq 0$. So $\xi_I(M) = 1$,            by \ref{power-of-I}. \end{proof}
\begin{remark}
If $\depth G_I(A) \geq d -1$ then by a result of Marley \cite[Corollary 2]{Marley-89}, all Hilbert coefficients of $A$ \wrt \ $I$ are non-negative.
Since $$e_3^{I^n}(M) = e_3^I(A) \quad \text{for all} \ \ n \geq 1$$
we get that if $e_3^I(A) < 0$ then $\depth G_{I^n}(M) \leq 1$ for all $n \gg 0$. It follows that
$\xi_I(M) =1$.
\end{remark}

We give the following example due to Marley   \cite[page 8]{Marley-89}
\begin{example}
Let $ A = \mathbb{Q}[x,y,x]$ and let $I = (x^3,y^3,z^3, x^2y, xy^2, yz^2, xyz)$. Then
$J = (x^3,y^3,z^3)$ is a minimal reduction of $I$. Furthermore $\red_I(J) = 2$. Using COCOA
one can check that $e_3(I) = -1$.
 By above remark $\xi_I(M) =1$. Notice $\red(I) \neq 1$ So it is  2 since   $\red_I(J) = 2$.
\end{example}

%\bibliographystyle{amsplain}
%\bibliography{rr2Ref}

\providecommand{\bysame}{\leavevmode\hbox to3em{\hrulefill}\thinspace}
\providecommand{\MR}{\relax\ifhmode\unskip\space\fi MR }
\providecommand{\MRhref}[2]{%
  \href{http://www.ams.org/mathscinet-getitem?mr=#1}{#2}
}
\providecommand{\href}[2]{#2}
 \end{document}